\documentclass[12pt,reqno]{article}
\usepackage{tikz}
\usepackage{mdwlist}
\usetikzlibrary{positioning}
\usetikzlibrary{decorations.text}
\usetikzlibrary{decorations.pathmorphing}
\usetikzlibrary{matrix,arrows,decorations.pathmorphing}
\usepackage{url}
\usepackage{hyperref}
\usepackage{empheq}
\usepackage{enumerate}
\usepackage[doc]{optional}
\usepackage{xcolor}
\usepackage{float}
\usepackage{soul}
\usepackage{graphicx}
\definecolor{labelkey}{rgb}{0,0.08,0.45}
\definecolor{refkey}{rgb}{0,0.6,0.0}
\definecolor{Brown}{rgb}{0.45,0.0,0.05}
\definecolor{lime}{rgb}{0.00,0.8,0.0}
\definecolor{lblue}{rgb}{0.5,0.5,0.99}

 \usepackage{mathpazo}

\usepackage{amsmath}

\usepackage{stmaryrd}
\usepackage{amssymb}
\usepackage{amsthm}

\newcommand{\sepp}{\setlength{\itemsep}{-2pt}}

\newcommand{\nnn}{\ensuremath{{n\in{\mathbb N}}}}

\newcommand{\menge}[2]{\big\{{#1}~\big |~{#2}\big\}}

\newcommand{\fenv}[1]%
{\ensuremath{\,\overrightarrow{\operatorname{env}}_{#1}}}
\newcommand{\benv}[1]%
{\ensuremath{\,\overleftarrow{\operatorname{env}}_{#1}}}

\newcommand{\RR}{\ensuremath{\mathbb R}}

\newcommand{\NN}{\ensuremath{\mathbb N}}

\newcommand{\SE}{\ensuremath{{\mathcal S}}}

\newcommand{\ran}{\ensuremath{\operatorname{ran}}}

\newcommand{\Id}{\ensuremath{\operatorname{Id}}}

\newcommand{\TBA}{T_{B,A}}
\newcommand{\TAB}{T_{A,B}}

{\begin{list}{}{%
\settowidth{\labelwidth}{\textrm{#1~}}%
\setlength{\leftmargin}{\labelwidth+\labelsep}}}
{\end{list}}

\usepackage{amsthm}
\usepackage[capitalize]{cleveref}
\crefname{equation}{}{equations}
\crefname{chapter}{Appendix}{chapters}
\crefname{item}{}{items}
\newtheorem{theorem}{Theorem}[section]

\newtheorem{lem}[theorem]{Lemma}

\newtheorem{cor}[theorem]{Corollary}

\newtheorem{prop}[theorem]{Proposition}

\newtheorem{thm}[theorem]{Theorem}
\newtheorem{example}[theorem]{Example}
\newtheorem{ex}[theorem]{Example}

\newtheorem{rem}[theorem]{Remark}
\theoremstyle{remark}

\providecommand{\abs}[1]{\lvert#1\rvert}
\providecommand{\norm}[1]{\lVert#1\rVert}

\providecommand{\bk}[1]{\left(#1\right)}
\providecommand{\minibk}[1]{(#1)}
\providecommand{\stb}[1]{\left\{#1\right\}}
\providecommand{\innp}[1]{\langle#1\rangle}

\providecommand{\RA}{\Rightarrow}

\providecommand{\RR}{\mathbb{R}}

\providecommand{\ran}{\operatorname{ran}}

\newcommand{\fix}{\ensuremath{\operatorname{Fix}}}

\providecommand{\gra}{\operatorname{gra}}
\providecommand{\Id}{\operatorname{{ Id}}}

\providecommand{\fady}{\varnothing}

\providecommand{\rras}{\rightrightarrows}

\providecommand{\NN}{\mathbb{N}}

\providecommand{\fix}{\operatorname{Fix}}
\providecommand{\ran}{\operatorname{ran}}

\providecommand{\Id}{\operatorname{Id}}

\providecommand{\nnn}{{n\in\NN}}

\providecommand{\fady}{\varnothing}

\providecommand{\ri}{\operatorname{ri}}

\providecommand{\RR}{\mathbb{R}}
\providecommand{\NN}{\mathbb{N}}

\definecolor{myblue}{rgb}{.8, .8, 1}
  \newcommand*\mybluebox[1]{%
    \colorbox{myblue}{\hspace{1em}#1\hspace{1em}}}

\allowdisplaybreaks 
\begin{document}
%
\title{\textsc
On the order of the operators in 
the Douglas--Rachford algorithm}
\author{
Heinz H.\ Bauschke\thanks{
Mathematics, University
of British Columbia,
Kelowna, B.C.\ V1V~1V7, Canada. E-mail:
\texttt{heinz.bauschke@ubc.ca}.}
~and Walaa M.\ Moursi\thanks{
Mathematics, University of
British Columbia,
Kelowna, B.C.\ V1V~1V7, Canada. E-mail:
\texttt{walaa.moursi@ubc.ca}.}}
\date{May 11, 2015}
\maketitle
\begin{abstract}
\noindent
The Douglas--Rachford algorithm is a popular method for finding
zeros of sums of monotone operators. 
By its definition, the Douglas--Rachford operator
is not symmetric with respect to the order 
of the two operators.
In this paper we provide 
a systematic study of the two 
possible Douglas--Rachford operators.
We show that the reflectors of the underlying operators 
act as bijections between the 
 fixed points sets of the two Douglas--Rachford operators.
Some elegant formulae arise under 
additional assumptions.
Various examples illustrate our results. 
\end{abstract}
{\small
\noindent
{\bfseries 2010 Mathematics Subject Classification:}
{Primary 
47H09, 
90C25. 
Secondary 
47H05, 
49M27, 
65K05. 
}

}
\section{Introduction}
Throughout this paper we shall assume that
$X$ is a real Hilbert space, 
with inner product $\innp{\cdot,\cdot}$ and
induced norm $\norm{\cdot}$. We also assume that  
$A:X\rras X$ and $B:X\rras X$ are maximally monotone operators\footnote{
Recall that 
$A:X\rras X$ is \emph{monotone} if whenever
the pairs 
$(x,u)$ and $(y,v)$ lie in $\gra A$
we have
$\innp{x-y,u-v}\ge 0$, and 
 is \emph{maximally monotone} 
if it is monotone and 
any proper enlargement of the graph of 
$A$ (in terms of set inclusion) 
does not preserve the monotonicity of $A$.}.
The \emph{resolvent} 
and the \emph{reflected resolvent}
associated with $A$ are $J_A=(\Id+A)^{-1}$
and 
$R_A=2J_A-\Id$, respectively\footnote{The identity operator on
$X$ is denoted by $\Id$. It is well-known that,
when $A$ is maximally monotone, 
$J_A$ is single-valued, maximally 
monotone and 
firmly nonexpansive
 and $R_A$ is nonexpansive.}. 
The sum problem for $A$ and $B$
is to \emph{find $x\in X$ such that $x\in (A+B)^{-1}0$}.
When $(A+B)^{-1}(0)\neq \fady $,
 the Douglas--Rachford splitting method
can be used to solve the sum problem.
The Douglas--Rachford 
splitting operator 
\cite{L-M79} associated with
the ordered pair of
operators $(A,B)$ is 
\begin{empheq}[box=\mybluebox]{equation}
\label{def:T}
\TAB:=\tfrac{1}{2}(\Id+R_BR_A)
=\Id-J_A+J_B R_A.
\end{empheq} 
By definition, the Douglas--Rachford 
splitting operator is dependent on the order
of the operators $A$ and $B$, even though the sum problem
remains unchanged
when interchanging $A$ and $B$.  
\emph{The goal of this paper is to investigate
the connection between the operators 
$\TAB$ and $\TBA$.}
Our main results
can be summarized as follows.
\begin{itemize}
\item
We show that $R_A$
is an isometric\footnote{Suppose that
$C$ and $D$ are two nonempty subsets of $X$.
We recall that
$Q:C\to D$ is an isometry 
if $(\forall x\in C)(\forall y\in C)$
 $\norm{Qx-Qy}=\norm{x-y}$.
 The set of fixed points of $T$ is
$\fix T:=\menge{x\in X}{x=Tx}$.} 
bijection from the fixed points set
of $\TAB$ to that of $ \TBA$,
with inverse $R_B:\fix \TBA\to\fix \TAB$
(see \cref{thm:mi}).
\item
When $A$ is an affine relation, we have
$(\forall\nnn)$ $R_A\TAB^n=\TBA^n R_A$.
In particular\footnote{
Throughout the paper we use 
$N_C$ and $P_C$ to denote 
the \emph{normal cone} and \emph{projector} associated
with a nonempty closed convex subset $C$ of $X$ respectively.},
when $A=N_U$ where $U$ is a closed affine subspace
of $X$, 
we have $(\forall \nnn)$
$\TAB^n=R_A\TBA^n R_A$ 
and $\TBA^n=R_A\TAB^n R_A$
(see \cref{prop:TAB:TBA}\ref{prop:TAB:TBA:ii} 
and \cref{cor:T:Ts}\ref{prop:TAB:TBA:iv}). 
\item
Our results connect to the recent 
linear and finite convergence 
results (see \cref{rem:big}) for the Douglas--Rachford
algorithm (see \cite{Ar-Br2013}, \cite{Ar-Br2014},
\cite{BDNP15}, \cite{BNP2014}, \cite{HL13} and \cite{HLN14}). 
\end{itemize}
In \cref{S:2},  we present the main results and various
examples. 
The notation we adopt is standard and follows, 
e.g., \cite{BC2011} and \cite{Rock98}.
\section{Results}
\label{S:2}
\mbox{}\par\nopagebreak\vspace{-2\baselineskip}\vspace{-
\abovedisplayskip}
\noindent
We recall that the Attouch--Th\'{e}ra dual pair of $(A,B)$
(see \cite{AT})
is the pair
\footnote{We set
$A^{\ovee}:= (-\Id)\circ A\circ(-\Id)$ and
$ A^{-\ovee}:=(A^{-1})^\ovee=(A^\ovee)^{-1}$.}
$(A^{-1},B^{-\ovee})$.
Following \cite{JAT2012}, we set 
$
Z:=Z_{(A,B)}=(A+B)^{-1}(0)
$
and 
$
K:=K_{(A,B)}=(A^{-1}+B^{-\ovee})^{-1}(0),
$ 
to denote, respectively, the primal and dual solutions. 
One easily verifies that 
\begin{equation}\label{eq:ZK:s}
Z_{(B,A)}=(B+A)^{-1}(0)=Z
\qquad\text{and }\qquad 
K_{(B,A)}=(B^{-1}+A^{-\ovee})^{-1}(0)=-K.
\end{equation}
We further recall 
(see \cite[Lemma~2.6(iii)]{Comb04} and \cite[Corollary~4.9]{JAT2012})
that
\begin{equation}\label{Fact:collect:Z:fix}
Z=J_A(\fix \TAB) \quad\text{and}\quad K=(\Id-J_A)(\fix \TAB),
\end{equation}
and we will make use of the following useful identity 
which can be verified using 
\cref{def:T}:
\begin{equation}
\label{prop:TAB:TBA:i}
R_A\TAB-\TBA R_A=2J_A\TAB-J_A-J_AR_BR_A.
\end{equation}

We are now ready for the first main result.
\begin{thm}\label{thm:mi}
$R_A$ is an isometric bijection 
from $\fix \TAB$ to $\fix \TBA$, 
with isometric inverse $R_B$. 
Moreover, we have the following 
commutative diagram:
\begin{center}
\begin{tikzpicture}[scale=6]
\node (A) at (0,0.75) {$\fix \TAB$};
\node (B) at (1,0.75) {$\fix \TBA$};
\node (C) at (0,0) {$\SE_{(A,B)}$};
\node (D) at (1,0) {$\SE_{(B,A)}$};
\draw
([yshift= 1pt]A.east) edge[->,>=angle 90] node[above]
{ $R_A$} ([yshift= 1pt]B.west)
([yshift= -1pt]A.east) edge[<-,>=angle 90] node[below]
{ $R_B$} ([yshift= -1pt]B.west);
\draw
([yshift= 0pt]C.east) edge[<->,>=angle 90] node[above]
{ $\Id\times(-\Id)$} ([yshift= 0pt]D.west);
\draw
[<-,postaction={decorate,decoration={text along path,raise=-10pt,
text align=center,text={${+}{\;}$}}}]
([xshift=-1pt,yshift=-3.5 pt]A.north) to node[swap] {} 
 ([xshift= -1pt,yshift=3.8 pt]C.south);
\draw[->,postaction={decorate,decoration={text along path, raise=5pt,text align=center,text={${(J_A,\Id-J_A)\circ \Delta}{\;}$}}}] 
([xshift=1pt,yshift=-3.5 pt]A.north) to node[swap] {} 
 ([xshift= 1pt,yshift=3.8 pt]C.south);
(A) to node[swap] {} (C);
\draw
[<-,
postaction={decorate,decoration={text along path,raise=-10pt,
text align=center,text={${+}{\;}$}}}]
([xshift=-1pt,yshift=-3.5 pt]B.north) to node[swap] {} 
 ([xshift= -1pt,yshift=3.8 pt]D.south);
B to node[swap] {} D;
    \draw[->,postaction={decorate,decoration={text along path, 
    raise=5pt,text align=center,
    text={${(J_B,\Id-J_B)\circ \Delta}{\;}$}}}] 
 ([xshift=1pt,yshift=-3.5 pt]B.north) to node[swap] {} 
 ([xshift= 1pt,yshift=3.8 pt]D.south);
\end{tikzpicture}
\end{center}
Here $\SE_{(A,B)}:=\stb{(z,-w)\in 
X\times X~|~-w\in Bz, w\in Az}$ is 
the extended solution set\footnote{For 
further information 
on the extended solution set,
we refer the reader to
\cite[Section~2.1]{Eck-Sv}.
} for the pair $(A,B)$, and
$\Delta\colon X\to X\times X \colon x\mapsto(x,x)$. 
In particular, 
we have 
\begin{equation}
\label{RA:para:case}
R_A\colon\fix\TAB\to \fix\TBA\colon z+k\mapsto z-k,
\end{equation}
where $(z,k)\in \SE_{(A,B)}$.
\end{thm}
\begin{proof}
Let $x\in X$ and note that \cref{def:T}
implies that $\fix\TAB=\fix R_BR_A$ and 
$\fix\TBA=\fix R_AR_B$. Now
$x\in \fix \TAB \iff x=R_BR_Ax
\RA R_A x=R_AR_BR_Ax
\iff R_A x\in \fix R_AR_B=\fix\TBA$,
which proves that $R_A$ maps 
$\fix \TAB$ into $\fix\TBA$.
By interchanging $A$
and $B$ one sees that
$R_B$ maps 
$\fix \TBA$ into $\fix\TAB$.
We now show that
$R_A$ maps $\fix \TAB$ onto $\fix\TBA$.
To this end, let $y\in \fix \TBA$ 
and note that
$R_B y \in \fix\TAB$ and
$R_A R_By=y$, which proves that
$R_A$ maps $\fix \TAB$ onto $\fix\TBA$.
The same argument holds for $R_B$.
Finally since $(\forall x\in \fix \TAB) $ 
$R_BR_A x=x$,  this proves that 
$R_A $ is a bijection from $\fix \TAB$
to $\fix \TBA$ with the desired inverse.
To prove that $R_A:\fix \TAB\to \fix \TBA$ is 
an isometry
note that $(\forall x\in \fix \TAB)$ $(\forall y\in \fix \TAB)$
 we have $\norm{x-y}=\norm{R_BR_Ax-R_BR_Ay}\le 
 \norm{R_Ax-R_Ay}\le\norm{x-y}$.
 
  We now turn to the diagram.
 The correspondence of $\fix \TAB$ and $\fix \TBA$
follows from 
our earlier argument.
The correspondences of $\fix \TAB$ 
and $\SE_{(A,B)}$, and $\fix \TBA$ 
and $\SE_{(B,A)}$
follow from combining
\cite[Remark~3.9 and Theorem~4.5]{JAT2012}
applied to $\TAB$ and $\TBA$ respectively.
The fourth correspondence
is obvious from the definition 
of $\SE_{(A,B)}$
and $\SE_{(B,A)}$.
To prove \cref{RA:para:case}
we let $y\in \fix \TAB$ and recall that in view of 
\cite[Theorem~4.5 and Remark~3.9]{JAT2012}
that
$y=z+k$ where $(z,k)\in \SE_{(A,B)}$
and
$R_A(z+k)=(J_A-(\Id-J_A))(z+k)
=J_A(z+k)-(\Id-J_A)(z+k)=z-k$.
\end{proof}
\begin{rem}
In view of
{\rm\cite[Remark~3.9, 
Theorem~4.5 and Corollary~5.5(iii)]{JAT2012}},
when $A$ and $B$ are
paramonotone\footnote{
See \cite{Iusem98} for definition 
and detailed discussion on
paramonotone operators.} 
(as is always the case
when $A$ and $B$ are subdifferential operators
of proper convex lower semicontinuous
functions),
we can replace $\SE_{(A,B)}$
and  $\SE_{(B,A)}$ by, respectively,
$Z\times K$ and $Z\times (-K)$. 
\end{rem}
\begin{lem}\label{lem:RP}
Suppose that $A$ is an affine relation.
Then
\begin{enumerate}
\item\label{lem:RP:0}
$J_A$
is affine and
$J_AR_A=2J_A^2-J_A=R_AJ_A$.
\suspend{enumerate}
If $A=N_U$, where 
$U$ is a closed affine subspace of $X$, then
we have additionally:
\resume{enumerate}
\item\label{lem:RP:i}
$P_U=J_A=J_AR_A=R_AJ_A$
 and 
$(\Id-J_A)R_A=J_A-\Id$.
\item\label{lem:RP:ii}
$R_A^2=\Id$, $R_A=R_A^{-1}$, and 
$R_A:X\to X$ is an isometric bijection.
\end{enumerate}
\end{lem}
\begin{proof}
\ref{lem:RP:0}:
The fact that $J_A$ is affine follows from 
\cite[Theorem~2.1(xix)]{BMW2012}.
Hence $J_AR_A=J_A(2J_A-\Id)
=2J_A^2-J_A=R_AJ_A$.

\ref{lem:RP:i}: It follows from 
\cite[Example~23.4]{BC2011}
that $P_U=J_A$. Now using 
\ref{lem:RP:0} we have
$R_AJ_A=J_AR_A
=2P_U^2-P_U=2P_U-P_U=P_U=J_A$.
To prove the last identity note that 
 by
\ref{lem:RP:0} we have 
$(\Id-J_A)R_A
=R_A-J_AR_A=2P_U-\Id+P_U
=P_U-\Id
$.

\ref{lem:RP:ii}:
Because $R_A$
is affine,
it follows from \ref{lem:RP:i}
that
$
R_A^2=R_A(2J_A-\Id)=2R_AJ_A-R_A
=2P_U-2(P_U-\Id)=\Id.
$
Finally let $x,y\in X$. Since $R_A$
is nonexpansive we have 
$\norm{x-y}=\norm{R_A^2x-R_A^2y}
\le \norm{{R_A}x-{R_A}y}\le  \norm{x-y}$,
hence all the inequalities 
become equalities which 
completes the proof.
\end{proof}

We now turn to the iterates of the 
Douglas--Rachford algorithm.
\begin{prop}\label{prop:TAB:TBA}
Suppose that $A$ is an affine relation. 
Then the following hold:
\begin{enumerate}
\item\label{prop:TAB:TBA:ii}
$(\forall n\in \NN)$ we have
$
R_A\TAB^n=\TBA^n R_A.
$
\item\label{prop:TAB:TBA:iii}
$
R_A Z=J_A\fix \TBA$ and  
$R_A K=(J_A-\Id)(-\fix \TBA).
$
\suspend{enumerate}
If $B$ is an affine relation, then we additionally have:
\resume{enumerate}
\item\label{prop:TAB:TBA:vi}
$\TAB R_BR_A=R_BR_A \TAB$.
\item\label{prop:TAB:TBA:viv}
$4(\TAB\TBA- \TBA\TAB)=R_BR_A^2R_B-R_AR_B^2R_A$.
Consequently,
$
\TAB\TBA=
\TBA\TAB \iff R_BR_A^2R_B=R_AR_B^2R_A.
$
\item\label{prop:TAB:TBA:vv}
If $R_A^2=R_B^2=\Id$, then\footnote{
 In passing, we point out that this is equivalent to saying that
 $A=N_U$ and $B=N_V$ where 
 $U$ and $V$ are closed affine subspaces of $X$. 
 Indeed, 
 $R_A^2=\Id\iff J_A=J_A^2$
 and therefore we conclude 
 that $\ran J_A=\fix J_A$.
Combining with
 \cite[Theorem~1.2]{Zar71:1}
 yields that $J_A$ is a projection, hence
 $A$ is an affine normal cone operator
 using \cite[Example~23.4]{BC2011}.
 }
$
\TAB\TBA=
\TBA\TAB .
$
\end{enumerate}
\end{prop}
\begin{proof}
\ref{prop:TAB:TBA:ii}: 
It follows from 
\cref{prop:TAB:TBA:i}, 
\cref{lem:RP}\ref{lem:RP:0}
and \cref{def:T}
that $R_A\TAB-\TBA R_A=2J_A\TAB-J_A-J_AR_BR_A
=J_A(2\TAB-\Id)-J_AR_BR_A
=J_A(2(\tfrac{1}{2}(\Id+R_BR_A))-\Id)
-J_AR_BR_A=J_AR_BR_A-J_AR_BR_A=0$,
which proves the claim when $n=1$.
The general proof follows by induction.

\ref{prop:TAB:TBA:iii}:
Using \cref{Fact:collect:Z:fix},
\cref{lem:RP}\ref{lem:RP:0} 
and \cref{thm:mi},
we have $R_AZ=R_A J_A(\fix \TAB)
=J_A R_A(\fix \TAB)=J_A (\fix \TBA)$.
 Now using that
 the inverse resolvent identity\footnote{
Recall the when $A$ is maximally monotone the 
inverse resolvent identity states that
$J_A+J_{A^{-1}}=\Id$.
Consequently, $R_{A^{-1}}=-R_A$.
},
 \cref{lem:RP}\ref{lem:RP:0}
 applied to $A^{-1}$ and \cref{thm:mi}, 
 we obtain
 $R_AK=-R_{A^{-1}}J_{A^{-1}}(\fix \TAB)=
 -J_{A^{-1}}R_{A^{-1}}(\fix \TAB)
 = -J_{A^{-1}}(-R_A\fix \TAB)
 =(J_A-\Id)(-\fix \TBA)
 $.

 \ref{prop:TAB:TBA:vi}:
Note that $\TAB$ and
$\TBA$ are affine. 
It follows from \cref{def:T} that
$\TAB R_BR_A=\TAB(2\TAB-\Id)
=2\TAB^2-\TAB=(2\TAB-\Id)\TAB
=R_BR_A\TAB$. 

\ref{prop:TAB:TBA:viv}
We have
\begin{align}
4(\TAB\TBA-\TBA\TAB)
 &=4 \bk{\tfrac{1}{2}(\Id+R_BR_A)\tfrac{1}{2}(\Id+R_AR_B)
-\tfrac{1}{2}(\Id+R_AR_B)\tfrac{1}{2}(\Id+R_BR_A)}
\nonumber\\
 &=\Id+R_BR_A+R_AR_B+R_BR_A^2R_B
-(\Id+R_AR_B+R_BR_A
\nonumber\\
&+R_AR_B^2R_A)= R_BR_A^2R_B-R_AR_B^2R_A.
\end{align}
\ref{prop:TAB:TBA:vv}: This is a direct consequence
of \ref{prop:TAB:TBA:vi}. 
\end{proof}

With regards to \cref{prop:TAB:TBA}\ref{prop:TAB:TBA:ii}, 
one may inquire whether the conclusion still
holds when $R_A$ is replaced by $R_B$.
We now give an example illustrating that
the answer to this question is negative. 
\begin{example}
Suppose that $X=\RR^2$, that $U=\RR\times\stb{0}$,
that $V=\stb{0}\times \RR_{+}$, that $A=N_U$
and that $B=N_V$.
Then $A$ is linear, hence $R_A\TAB=\TBA R_A$, however 
$R_B\TAB\neq\TBA R_B$ and $R_B\TBA\neq\TAB R_B$. 
\end{example}
\begin{proof}
The identity $R_A\TAB=\TBA R_A$
follows from applying 
\cref{prop:TAB:TBA}\ref{prop:TAB:TBA:ii}
with $n=1$. Now let $(x,y)\in \RR^2$.
Elementary calculations show that
$R_A(x,y)=(x,-y)$.
and $R_B(x,y)=(-x,\abs{y})$.
Consequently, 
\cref{def:T} implies that
$\TAB(x,y)=(0,y^+)$
and
$\TBA(x,y)=(0,y^{-})$
\footnote{For every $x\in\RR$, we set
$x^+:=\max\{x,0\}$ and $x^{-}:=\min\{x,0\}$}.
Therefore, 
$R_B\TAB(x,y)=(0,y^+)$,
$\TBA R_B(x,y)= (0,0),
$
$
R_B\TBA(x,y) =(0,{-y}^+)$,
and $
\TAB R_B(x,y)= (0, \abs{y}).
$
The conclusion then follows from comparing
the last four equations.
\end{proof}

We are now ready for our second main result.
\begin{thm}[\bf{When $A$ is normal cone 
of closed affine subspace}]
\label{cor:T:Ts}
Suppose that $U$ is a closed affine subspace
and that $A=N_U$. Then the following hold:
\begin{enumerate}
\item\label{prop:TAB:TBA:iv}
$(\forall\nnn)$ $
R_A\TBA^n=\TAB^n R_A
$,
$
\TBA^n=R_A\TAB^n R_A
$
and
$
\TAB^n=R_A\TBA^n R_A
$.
\item\label{prop:TAB:TBA:v} 
$R_A\colon\fix \TBA\to \fix \TAB$, 
$
Z=J_A(\fix\TBA)$, and $ K=(J_A-\Id)(\fix\TBA)$. 
\item\label{prop:TAB:TBA:vii}
Suppose that 
$V$ is a closed affine subspace
of $X$
and that $B=N_V$.
Then 
$\TAB R_AR_B=R_AR_B \TAB$
and
$\TAB\TBA=\TBA\TAB$.
\end{enumerate}
\end{thm}
\begin{proof}
\ref{prop:TAB:TBA:iv}:
Let $\nnn$.
It follows from 
\cref{prop:TAB:TBA}\ref{prop:TAB:TBA:ii} 
and \cref{lem:RP}\ref{lem:RP:ii} that  
$\TAB^n=R_AR_A\TAB^n=R_A\TBA^n R_A$.
Hence $\TAB^n R_A
=R_A\TBA^n R_AR_A=R_A\TBA^n$.

\ref{prop:TAB:TBA:v}:
The statement for $R_A$
follows from combining \cref{thm:mi}
and \cref{lem:RP}\ref{lem:RP:ii}.
In view of \cref{Fact:collect:Z:fix},
\cref{lem:RP}\ref{lem:RP:i} and \cref{thm:mi}
one learns that
$Z= J_A (\fix\TAB)=J_AR_A(\fix\TAB)=J_A(\fix \TBA)$.
Finally, \cref{Fact:collect:Z:fix}, 
\cref{lem:RP}\ref{lem:RP:ii} and \ref{lem:RP:i}, 
and \cref{thm:mi}
imply 
that $K=(\Id-J_A)(\fix \TAB)=(\Id-J_A)R_A (R_A \fix \TAB)
=(J_A-\Id)\fix \TBA$.

\ref{prop:TAB:TBA:vii}:
In view of \ref{prop:TAB:TBA:ii}
applied to $A$ and $B$
we have 
$\TAB R_AR_B=R_A\TBA R_B=R_AR_B\TAB$.
The second identity follows from combining
\cref{prop:TAB:TBA}\ref{prop:TAB:TBA:vv} 
and \cref{lem:RP}\ref{lem:RP:ii} applied
to both $A$ and $B$.
\mbox{}\par\nopagebreak\vspace{-0.05\baselineskip}
\vspace{-\abovedisplayskip}
\end{proof}

\begin{figure}[h!]
\begin{center}
\begin{tabular}{c c c}
\includegraphics[scale=0.28]{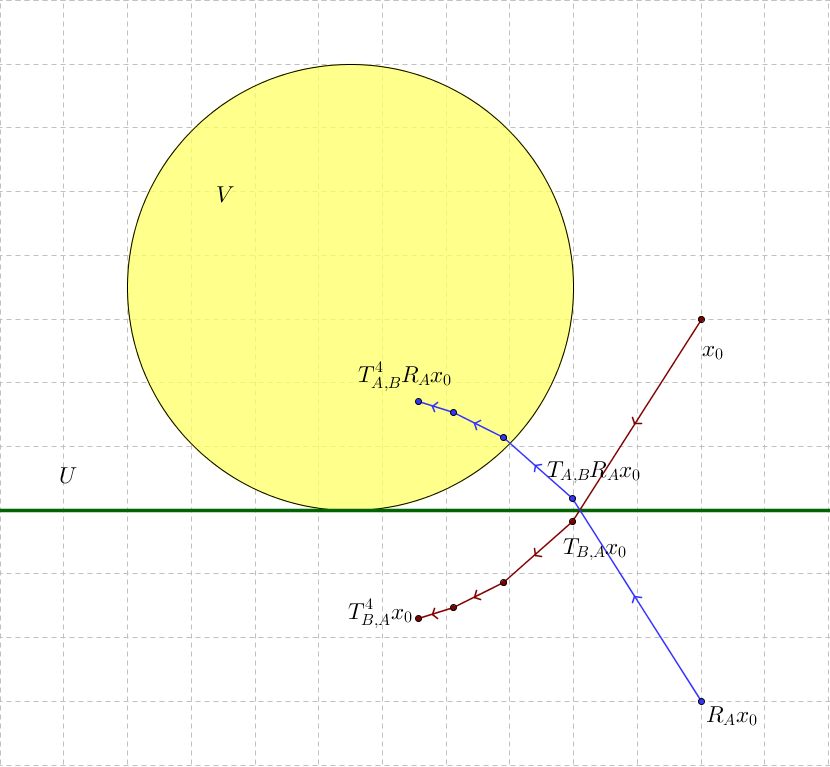}    & &
\includegraphics[scale=0.28]{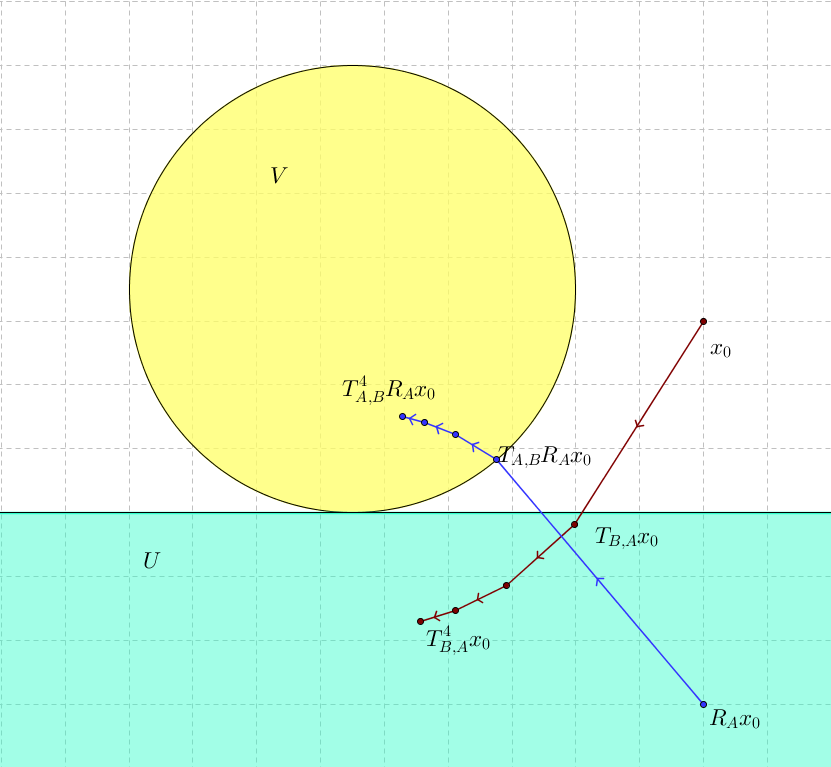}   
\end{tabular}
\end{center}
 \caption{A \texttt{GeoGebra} \cite{geogebra} snapshot.
 Left: Two closed convex sets in $\RR^2$,
 $U$ is a linear subspace (green line)
  and $V$ (the ball).
  Right: Two closed convex sets in $\RR^2$,
 $U$ is the halfspace (cyan region)
 and $V$ (the ball).
 Shown are also the first five terms of the sequences
 $(\TAB^n R_Ax_0)_{n\in \NN}$ (red points) and 
  $(\TBA^n x_0)_{n\in \NN}$ (blue points)
  in each case. The left figure illustrates
  \cref{cor:T:Ts}\ref{prop:TAB:TBA:iv} while the right figure
  illustrates the failure of this result when the subspace is
  replaced by a cone.}
\label{fig:two:set}
\end{figure}
The conclusion of 
\cref{cor:T:Ts}\ref{prop:TAB:TBA:vii} 
may fail when we assume that $A$ or
$B$ is an affine, but not a normal cone, operator
as we illustrate next.
\begin{example}
Suppose that $X=\RR^2$, 
that $U=\RR\times\stb{0}$,
that $A=N_U$
and that
$B = \left(\begin{smallmatrix} 1&1\\
1&1\\ \end{smallmatrix}\right)$.

Then $B$ is linear and maximally monotone
\emph{but not a normal cone} operator
 and 
 $
\tfrac{1}{9}
\left(\begin{smallmatrix} 5&-1\\
-1&2\\ \end{smallmatrix}\right)
=
\TAB\TBA\neq 
\TBA\TAB
=\tfrac{1}{9}
\left(\begin{smallmatrix}
5&1\\
1&2\\
\end{smallmatrix}\right)
$
\end{example}

\begin{cor}\label{cor:NU:B:it}
Suppose that $U$ is an affine subspace,
that $A=N_U$ and that $Z\neq \fady$. 
Let $x$ and $y$ be in $X$.
Then the following hold:
\begin{enumerate}
\item\label{cor:NU:B:it:i}
$(\forall\nnn)$ 
$J_A \TBA^n x=J_A \TAB^n R_Ax $, and 
$(J_A \TBA^n x)_{n\in\NN}$ 
converges weakly to a point in $Z$.

\item\label{cor:NU:B:it:iii}
$\norm{\TAB x-\TAB y}
=\norm{\TBA R_Ax-\TBA R_A y}\le 
\norm{ R_Ax-R_A y}$.
\end{enumerate}
\end{cor}
\mbox{}\par\nopagebreak\vspace{-1\baselineskip}\vspace{-
\abovedisplayskip}
\begin{proof}
\ref{cor:NU:B:it:i}:
It follows from \cref{lem:RP}\ref{lem:RP:ii}, 
\cref{prop:TAB:TBA}\ref{prop:TAB:TBA:ii}
and  \cref{lem:RP}\ref{lem:RP:i} that
$J_A \TBA^n x
=J_A  \TBA^n R_A R_Ax
=J_A R_A\TAB^n  R_Ax
=J_A\TAB^n  R_Ax$,
as claimed.
The convergence of the sequence 
$(J_A \TBA^n x)_{\nnn}$ follows from
e.g., \cite[Theorem~25.6]{BC2011}.

\ref{cor:NU:B:it:iii}:
Apply 
\cref{lem:RP}\ref{lem:RP:ii} 
with $x$
and $y$ replaced with $\TAB x$
and $\TAB y$, 
\cref{prop:TAB:TBA}\ref{prop:TAB:TBA:ii}
with $n=1$, and use 
nonexpansiveness of $\TBA$.
\end{proof}
\begin{rem}\label{rem:big}
\ 
\begin{enumerate}
\item\label{R1}
The results of {\rm\cref{cor:T:Ts}} 
and {\rm\cref{cor:NU:B:it}} are of interest when 
the Douglas--Rachford method is applied to 
find the zero of the sum of more than two operators in which case 
one can use 
a {parallel splitting method} 
{\rm (see e.g., \cite[Proposition~25.7]{BC2011})},
where one operator is the normal 
cone operator of the diagonal subspace in a product space. 
\item\label{R2}
A second glance at the proof of {\rm\cref{cor:T:Ts}\ref{prop:TAB:TBA:ii}}
reveals that the result remains true if 
$J_B$ is replaced by any operator $Q_B \colon X \to X$
(and $R_B$ is replaced by $2Q_B-\Id$, of course). 
This is interesting because in {\rm\cite{Ar-Br2013}}, 
 {\rm\cite{Ar-Br2014}},
{\rm\cite{HL13}} and {\rm\cite{HLN14}},
$Q_B$ is chosen to be a selection of the (set-valued)
projector onto a set $V$ that is \emph{not convex}. 
Hence the generalized variant of 
{\rm\cref{cor:T:Ts}\ref{prop:TAB:TBA:ii}}
 then
guarantees that the orbits of the two Douglas--Rachford operators
are related via 
\begin{equation}
(\forall\nnn)\quad \TBA^n=R_A\TAB^n R_A.
\end{equation}

\item\label{R3}
As a consequence of {\rm\ref{R2}} and
{\rm\cref{lem:RP}\ref{lem:RP:ii}},
we see that if 
\emph{linear convergence} is guaranteed for the iterates of $\TAB$
then the same holds true for 
the iterates of $\TBA$ provided that 
$U$ is a closed affine subspace,
$V$ is a nonempty closed set,
$A=N_U$ and $J_B$ is a selection of the projection onto $V$. 
This is not particularly striking when we compare to sufficient
conditions that are already 
symmetric in A and B (such as, e.g., $\ri U \cap \ri V\neq \fady$
 in {\rm\cite{BNP2014}} and {\rm\cite{Phan2014}});
however, this is a new insight when the sufficient conditions are 
\emph{not symmetric} (as in, e.g., {\rm\cite{Ar-Br2013}}, {\rm\cite{B-Sims2011}}
{\rm\cite{HL13}} and {\rm\cite{HLN14}}). 

\item 
A comment similar to {\rm\ref{R3}} can be made for
\emph{finite convergence} results; see {\rm\cite{Spin85}} and  
{\rm\cite{BDNP15}} for
\emph{nonsymmetric} sufficient conditions. 
  \end{enumerate}
\end{rem}

We now turn to the Borwein--Tam method \cite{B-Tam}. 
\begin{prop}\label{Prop:BT}
Suppose that $U$ is an affine
subspace of $X$, that $A=N_U$, and set 
\begin{equation}\label{e:def:CDR}
T_{[A,B]}:=\TAB\TBA.
\end{equation}
\mbox{}\par\nopagebreak\vspace{-1.5\baselineskip}\vspace{-
\abovedisplayskip}
Then the following holds: 
\mbox{}\par\nopagebreak\vspace{-0.5\baselineskip}\vspace{-
\abovedisplayskip}
\begin{enumerate}
\item\label{Prop:BT:i}
$T_{[A,B]}=R_AT_{[B,A]}R_A
=(\TAB R_A)^2=(R_A\TBA)^2.$
\item\label{Prop:BT:ii}
Suppose that $V$ is an affine
subspace and that $B=N_V$.
Then\footnote{See \cite[Proposition~3.5]{JAT2014} for
the case when $U$ and $V$ are linear subspaces.}
 $T_{[A,B]}=T_{[B,A]}$.
Consequently
$T_{[A,B]}=(R_B\TAB)^2=(\TBA R_B)^2
=\tfrac{1}{2}(\TAB+\TBA)$, and
$T_{[A,B]}$ is firmly nonexpansive.
\end{enumerate}
\end{prop}
\begin{proof}
\ref{Prop:BT:i}:
Using \cref{e:def:CDR} and
\cref{cor:T:Ts}\ref{prop:TAB:TBA:iv}
with $n=1$ we obtain 
$
T_{[A,B]}
=R_A\TBA R_A \TBA=R_A\TBA \TAB R_A=
R_AT_{[B~A]} R_A
=\TAB R_A\TAB R_A
=(\TAB R_A)^2
=(R_A\TBA)^2.
$

\ref{Prop:BT:ii}:
The identity $T_{[A,B]}=T_{[B,A]}$ 
follows from 
\cref{cor:T:Ts}\ref{prop:TAB:TBA:vii}
 and \cref{e:def:CDR}. 
Now combine with \ref{Prop:BT:i}
with $A$ and $B $ switched, and use
\cite[Remark~4.1]{B-Tam}.
That 
 $T_{[A,B]}$ (hence $T_{[B,A]}$)
 is firmly nonexpansive follows from 
 the firm nonexpansiveness of 
 $\TAB$ and $\TBA$ and
 the fact
 that the class of firmly nonexpansive
 operators is closed under convex combinations
 (see, e.g., \cite[Example~4.31]{BC2011}). 
\end{proof}

Following {\rm\cite{B-Tam}}, the Borwein--Tam method
specialized to two nonempty closed convex subsets $U$
and $V$ of $X$, iterates the operator 
$T_{[A,B]}$ of \eqref{e:def:CDR}, 
where $A=N_U$ and $B=N_V$.
We conclude with an example that
shows that if $A$ or $B$ is not an
affine normal cone operator 
then $T_{[A,B]}$
and $T_{[B,A]}$
need not be firmly nonexpansive.
\begin{example}
Suppose that $X=\RR^2$,
that $U=\RR_+\cdot(1,1)$,
that $V=\RR\times \stb{0}$,
that $A=N_U$
and that
$B=N_V$.
Then neither $T_{[A,B]}$
nor $T_{[B,A]}$ is
firmly nonexpansive.
\end{example}
\mbox{}\par\nopagebreak\vspace{-1.5\baselineskip}\vspace{-
\abovedisplayskip}
\begin{proof}
Let $(x,y)\in \RR^2$.
Using \cref{def:T} we verify that 
$\TAB(x,y)=
\minibk{\tfrac{1}{2}(x+y)^+,y-\tfrac{1}{2}(x+y)^+}$
and 
$
\TBA(x,y)= \minibk{\tfrac{1}{2}(x-y)^+
,y+\tfrac{1}{2}(x-y)^+}$.
Now let $\alpha>0$, let $x=(-2\alpha,2\alpha)$
and let $y=(0,0)$.
A routine calculation shows that 
$T_{[A,B]}x=\TAB\TBA (-2\alpha,2\alpha)=(\alpha,\alpha)$
and $T_{[A,B]}y=\TAB\TBA (0,0)=(0,0)$, hence
$\innp{T_{[A,B]}x-T_{[A,B]}y, (\Id-T_{[A,B]})x-(\Id-T_{[A,B]})y}
=\innp{(\alpha,\alpha), (-3\alpha,\alpha)}=-2\alpha^2<0$. 
Applying similar argument to 
$T_{[B,A]}$ with $x=(-2\alpha,-2\alpha)$
and $y=(0,0)$ shows that
$\innp{T_{[B,A]}x-T_{[B,A]}y, (\Id-T_{[B,A]})x-(\Id-T_{[B,A]})y}
=\innp{(-3\alpha,-\alpha),(\alpha,-\alpha)}=
-2\alpha^2<0$.
It then follows from e.g., \cite[Proposition~4.2]{BC2011}
that neither $T_{[A,B]}$ nor $T_{[B,A]}$ is firmly nonexpansive.
\end{proof}

\vskip 8mm

\mbox{}\par\nopagebreak\vspace{-3\baselineskip}\vspace{-
\abovedisplayskip}

\end{document}